\documentclass[a4paper,12pt]{amsart}
\usepackage{amsmath,mathabx,mathtools}
\usepackage[margin=2cm,a4paper]{geometry}
\usepackage{dsfont}
\usepackage{eulervm}

\usepackage{tikz}
\theoremstyle{definition}

\newtheorem{Cor}{Corollary}

\newtheorem{Prop}{Proposition}

\newtheorem{Cond}{Assumption}
\newtheorem{example}{Example}%[chapter]
\newtheorem{remark}{Remark}%[chapter]
\newtheorem{theorem}{Theorem}%[chapter]
\DeclarePairedDelimiter\floor{\lfloor}{\rfloor}

\newcommand{\abs}[1]{\left\vert#1\right\vert}

\newcommand{\Su}{S}

\newcommand{\Cof}{{\Xi}}
\DeclareMathOperator{\Part}{\mathcal{P}}

\newcommand{\id}{\mathrm{Id}}

\newcommand{\rc}{\mathrm{Cr} }

\def\state{\varphi }
\def\A{\mathcal{A}}
\def\M{{\state}_{N,\pi}}
\def\T{ \mathbf{T}}

\def\E{\mathbb{E}}
\def\mutylda{\widetilde{Q}}

\def\muh{Q}

\def\C{{\mathbb C}}
\def\D{B_{\alpha,q}}
\def\Pro{{\mathbb P}}

\def\R{{\mathbb R}}
\def\N{{\mathbb N}}

\def\K{\mathcal{K}}

\def\d{{\rm d}}
\def\6{\, {\rm d}}
\def\1{1}
\def\i{{\rm i}}
\def\ri{{\rm i}}

\def\B{b_{\alpha,q}}
\def\G{G_{\alpha,q}}
\def\L{L_{\Krop}}

\def\Y{ \mathrm{Y}_{N,\pi_f}}
\def\YM{ \mathrm{Y}_{M,\pi_f}}

\def\r{r}
\def\Krop{\Delta}
\def\F{\mathcal{F}_{\rm fin}(H)}

\def\id{I}

\def\qMP{{\rm MP}}
\def\P{\mathcal{P}}

\def\PB{\mathcal{P}^B}

\def\InNB{\text{\normalfont Nest}}
\def\NB{\text{\normalfont NB}}

\newcommand{\e}{{\epsilon}}

\usepackage{enumerate}

\numberwithin{equation}{section}

\def\cvput#1[#2]{\pnode(#1,1){#1} \pscircle*(#1,1){.1} \rput(#1,.5){$#2$}}

\def\cvpuW#1[#2]{\pnode(#1,1){#1} \pscircle*(#1,1){.1} }

\def\cvpuDots#1[#2]{\pnode(#1,1){#1} \dots }

\newcommand{\Comment}[1]{}

\title[Central limit theorem  associated to Gaussian operators of type B]{Central limit theorem  associated to Gaussian operators of type B}
\author[Wiktor Ejsmont]{Wiktor Ejsmont}
\address { %Wydzia\l  Matematyki i Informatyki, Uniwersytet Im. Adama Mickiewicza, Collegium
Wroc\l aw University of Economics, ul.\ Komandorska 118/120, 53-345 Wroc\l aw, Poland}
\email{wiktor.ejsmont@gmail.com}
\begin{document}
\maketitle
\begin{abstract}
Speicher  \cite{S92}  showed a fundamental fact in noncommutative probability theories and generalized the Central Limit Theorem (CLT).  
This result  provides a very important tool to investigate noncommutative random variables. 
Nou \cite{N06} used  Speicher's result to provide Hiai's $q$-Araki–Woods von Neumann
algebras.
Also Biane \cite{Bi97} employed Speicher's central limit theorem to deduce Nelson's inequalities for
the functors $\Gamma_q$.
 In this article we formulate the CLT associated to Gaussian operators of type B -- see \cite{BEH15}, where important role is played by colored pair partitions. 
 Then we present a certain family of noncommutative random matrix models for the $(\alpha,q)$--deformed Gaussian random variables.
\end{abstract}
\section{Introduction}
\subsection{The deformed Gaussian variables of type B} \label{sectionmoment}
At the beginning we introduce essential
information to deal with Gaussian variables of type B (%its name comes from the
correspond to 
 Coxeter groups of type B; see \cite{BEH15} for more details). 
Let $H_\R$ be a separable real Hilbert space and let $H$ be its complexification with inner product $\langle\cdot,\cdot\rangle$, linear on the right component 
and anti-linear on the left. When considering elements in $H_\R$, it holds true that $\langle x,y\rangle=\langle y,x\rangle$.
We also assume that there exists a self-adjoint involution $H\ni x\mapsto \bar{x}\in H$, i.e. a self-adjoint linear bounded operator on $H$ such that the double 
application of it becomes the identity operator.
   Let $\F$ be the (algebraic) full Fock space over $H$ defined as 
$
\F:= \bigoplus_{n=0}^\infty H^{\otimes n} ,\label{fockspace}
$ 
with convention that $H^{\otimes 0}=\C\Omega$ is a one-dimensional normed unit vector called a vacuum. Note that elements of $\F$ are finite linear combinations 
of the elements from $H^{\otimes n}, n\in \N\cup\{0\}$ and we do not take the completion.
The \emph{Gaussian operator of type B} or \emph{$(\alpha,q)$--Gaussian operator} is 
 $$\G(x)= \B(x) +\B^\ast(x)\qquad x \in H,\text{ }\alpha,q\in (-1,1),$$ where operators $\B(x)$ and their adjoints $\B^\ast(x)$ fulfill deformed commutation
relations
\begin{equation*}
\B(x)\B^\ast(y)- q \B^\ast(y)\B(x)= \langle  x,y \rangle \id + \alpha \langle x, \bar{y} \rangle\, q^{2 N}. 
\end{equation*}
In the equation above $q^{2 N}$ is the operator on $\F$ defined by the linear extension of $q^{2 N}\Omega=1$ and  
$q^{2 N}x_1 \otimes \dots \otimes  x_n =q^{2 n} x_1 \otimes \dots \otimes  x_n$.
 These operators act on a Hilbert space $\F$, with the property that $\B(x)\Omega=0.$
This information is sufficient to compute the moments of the type-B Gaussian operator 
with respect to the vacuum vector state $\langle\Omega,\cdot\Omega\rangle$ (playing the role of expectation). In particular, we have the following result -- 
see \cite{BEH15};
for any $n\in \N$ and  $x_1,\dots,x_{2n+1} \in H_\R$ we have  
\begin{align*}
&\langle\Omega, \G(x_{1})\cdots \G(x_{2n+1})\Omega\rangle=0, \\
&\langle\Omega, \G(x_{1})\cdots \G(x_{2n})\Omega\rangle=
\displaystyle\sum_{\pi_f \in \PB_2(2n)} \alpha^{\NB(\pi_f)}q^{\rc(\pi)+2 \InNB(\pi_f)}\prod_{\substack{\{i,j\} \in \pi \\ f(\{i,j\})=1} }\langle x_i, x_j\rangle \prod_{\substack{\{i,j\} \in \pi\\ f(\{i,j\})=-1}} \langle x_i,\overline{x_j}\rangle,
\end{align*} 
where  $\pi_f$  are the pair partitions of type B with some special function on it  ($\NB,\rc,\InNB$) -- for the reader's
convenience we shall introduce this definition in Subsection  \ref{sect:partition}.
\subsection{Distribution}
A distribution of a random
variable corresponding to the bounded self-adjoint operator $X$ is a measure $\mu$ supported
on the real line $\R$ such that
\begin{align*}
&\langle\Omega,X^n\Omega\rangle=\int_\R x^n d\mu(x).
\end{align*} 
Let $\qMP_{\alpha,q}$ be the probability measure supported on $(-2/\sqrt{1-q}, 2/\sqrt{1-q})$, with absolutely continuous part given by 
\begin{equation}\label{eq10}
\frac{d \qMP_{\alpha,q}}{dt}(t)=  \frac{(q;q)_\infty(\beta^2; q)_\infty}{2\pi\sqrt{4/(1-q) -t^2}}\cdot\frac{g(t,1;q) g(t,-1;q) g(t,\sqrt{q};q) g(t,-\sqrt{q};q)}{g(t, \i \beta;q)g(t,-\i \beta;q)}
\end{equation}
where 
\begin{align*}
g(t,b;q)&= \prod_{k=0}^\infty(1-4 b t (1-q)^{-1/2} q^k + b^2 q^{2k}), \qquad 
(s;q)_\infty=\prod_{k=0}^\infty(1-s q^{k}),\text{ } s \in \R, \\
%=(2 \i  b e^{\i \theta}/\sqrt{1-q}, -2 \i b e^{\i \theta}/\sqrt{1-q};q)_\infty, \\ 
\beta
&= 
\begin{cases}
\sqrt{-\alpha}, & \alpha \leq 0, \\
\ri \sqrt{\alpha}, & \alpha \geq0.   
\end{cases}
\end{align*}
 If $\alpha,q\in(-1,1)$ and $x \in H, \|x\|=1$, then $\qMP_{\alpha\langle x,\bar{x}\rangle,q}$ is the probability distribution of $\G(x)$ with respect to 
 the vacuum state (by weak continuity we may allow the parameters $(\alpha,q)$ of $\qMP_{\alpha,q}$ to take any values in $[-1,1] \times [-1,1]$). 
\begin{example}\label{przyklad1} In the special case we have
\begin{enumerate}[\rm(1)]
\item The measure $\qMP_{\alpha,1}$ is the normal law $(2(1+\alpha)\pi)^{-1/2}e^{-\frac{t^2}{2(1+\alpha)}}1_\R(t)\,\d t$;

\item The measure $\qMP_{0,0}$ is the standard Wigner's semicircle law $(1/2\pi)\sqrt{4-t^2}1_{(-2,2)}(t)\,\d t$;

\item The measure $\qMP_{0,q}$ is the $q$--Gaussian law~\cite{BS91};

\item The measure $\qMP_{\alpha,-1}$ is the Bernoulli law 
$(1/2)(\delta_{\sqrt{1+\alpha}}+\delta_{-\sqrt{1+\alpha}})$;

\item The measure $\qMP_{\alpha,0}$ is a symmetric free Meixner law~\cite{A03}. 
\end{enumerate}
\end{example}
\subsection{Noncommutative Central Limit Theorem }
Our motivation to find random matrices which
asymptotically behave like $(\alpha,q)$--Gaussian variables were inspired by a careful
study of the article of Speicher \cite{S92}.
In \cite{S92}, Speicher showed a CLT for the special measure of Example \ref{przyklad1}, i.e. for the $q$--Gaussian law. Speicher's CLT concerns 
a sequence of elements $b_1,\dots,b_n$, whose
terms pair-wise satisfy the deformed commutation relation $b_ib_j = s(j, i)b_jb_i$ with $s(j, i) \in \{-1, 1\}$. It is not a priori
clear that the  sums $\frac{b_1+b_1^\ast+\dots+b_n+b_n^\ast}{\sqrt{N}}$ should converge in some reasonable sense for $q$--Gaussians, but that 
indeed turns out to be the case.
Later this topic was deeply analyzed by several autors  \cite{S01,K05,B14}. \'Sniady \cite{S01}  constructed a family of Gaussian random matrix (i.e.  
their entries are Gaussian) models for the $q$--Gaussian random variables.   Kemp \cite{K05}  obtained similar model for the corresponding $q$-deformed
circular system replacing $2\times 2$ matrices (Speicher approach)  by  $4\times 4$ block-diagonal matrices.   
 A related problem of finding a random matrix model for the so-called $(q,t)$--Gaussian measure (see \cite{B12}) was solved by Blitvi\'c \cite{B14}.
 
  At this point it is worth to mention that in noncommutative literature recently we can also find some other approximations of $q$--Gaussian measure by 
  Wigner integrals (see Deya, Aur\'elien, Noreddine, Nourdin \cite{DNN13}). The authors prove a fourth moment theorem for multiple integrals driven by 
  a $q$--Brownian motion. This finding extends the recent results by Nualart and Peccati \cite{NP05} to $q$--deformed probability theory. 
\subsection{Goal and overview of the paper}  \label{subgoal}
 In this paper we show a certain noncommutative central limit theorem asserting  that if a
suitably selected family of centered noncommutative random variables $\T_1^{\e(1)},\dots,\T_n^{\e(n)}$, $\e(i)\in\{1,\ast,\prime\}$
has the property that each pair of them satisfies the commutation relation if and only if these pairs are independent in a general sense, then the distribution 
of the normalized mean 
converges to the distribution of a Gaussian random variable of type B.
 The paper is organized as follows. First, we present definitions of type-B partitions and introduce special notation. In section \ref{sekcjaCLT} we bring 
 in additional restrictions on our extra operators and give an abstract version of type-B CLT. Finally, in section 4, we construct a central object of this 
 paper, namely random matrices, which satisfy the assumption of CLT.
 We also show that these matrices asymptotically have the same expectation values as $(\alpha,q)$--deformed Gaussian random variables.
 %and they converge in  sensible meaning to $(\alpha,q)$-Gaussian operators.

\section{Preliminaries } 
\subsection{Partitions and statistics }\label{sect:partition}
Let $[n]$ be the set $\{1,\dots,n\}$. For an ordered set $S$, let $\Part(S)$ denote the lattice of set partitions of that set. We write $B \in \pi$ if $B$ is 
a class of $\pi$ and we say that $B$ is a \emph{block of $\pi$}. We denote by $|\pi|$ the number of blocks of $\pi,$ and  $\#\pi:=\#S$ is the cardinality of $S$.  
A class of $\pi$ is called a \emph{singleton} if it consists of one element.
 \emph{A pair} (or a pair block) $V$ of a set partition is a block with cardinality 2. 
 We order the classes of $\pi=\{B_1,\dots,B_l\}$ according to the order of their last elements, i.e. $\max(B_1) < \max(B_2) < \ldots < \max(B_l)$. 
 %and a \emph{singleton} of a set partition is a block with cardinality 1. 
When $n$ is even, a set partition of $[n]$ is called a \emph{pair partition} if every block is a pair. The set of pair partitions of $[n]$ is denoted by $\P_2(n)$. 

The element  $\pi_f$  is called a \emph{set partition of  $[n]$  of type B} if  $\pi$  is a set partition of  $[n]$  and  $f:\pi\to \{\pm1\}$  is a coloring of 
the blocks of  $\pi$.  We denote by  $\PB(n)$  the set of all set partitions of  $[n]$  of type B. 
The notation  $\PB_{2}(n)$  denotes the set of set partitions of  $[n]$  of type B such that each block is a pair with color  $\pm1$. 
\newline
\newline
\noindent
%\textbf{Statistics.} 
Now, we introduce some partition statistics for $\pi_f\in\PB_2(n)$. 
In the equation below, we skip the index $f$ in $\pi$ if our statistic does not depend on coloring. 
Let $\NB(\pi_f)$ be the set of negative blocks (i.e.\ blocks colored by $-1$). 

\noindent For two blocks $V, W$ of a set partition $\pi\in \P_2(n)$, we introduce the relations  $\text{cr}\text{ and nest}$ as follows 
\begin{align*}
V\stackrel{\text{cr}}{\sim}W 
&\iff
\text{there exist $i,j \in V, k,l\in W$ such that $i<k<j<l$,}
\\
V\stackrel{\text{nest}}{\sim}W 
&\iff
\text{if there are $i,j \in V$ such that $i <k <j$ for any $k\in W.$}
\end{align*}
For a set partition $\pi$ let $\rc(\pi)$ be the number of crossings of $\pi$, i.e.\ 
\begin{align*}
\rc(\pi)&=\#\{\{V,W\} \subset \pi \mid V\stackrel{\text{cr}}{\sim}W \}. 
\intertext{Let $\InNB(\pi_f)$ be the number of pairs of a negative block and a nesting block}
\InNB(\pi_f)&=\#\{(V,W) \in \pi_f \times \pi_f \mid  f(W)=-1, V \stackrel{\text{nest}}{\sim}W  \}. 
\end{align*}
\subsection{Some special notations} 
Let 
$h:[r]\to\N$ be a map. We denote by $\ker h$ the set partition
which is induced by the equivalence relation 
$$
k\sim_{\ker h} l
\iff
h(k)=h(l)
.
$$
Similarly, for a multiindex $\underline{i} =( i(1),i(2),\dots ,i(n))\in\N^n$ we denote its
kernel
$\ker\underline{i}$ by the relation $k\sim l$ if and only if $i(k)=i(l)$.
Note that writing $\ker\underline{i}=\pi$ will indicate that $( i(1),i(2),\dots ,i(n))$ is in
the equivalence class identified with the partition $\pi\in\P(n)$.

\noindent Given $\epsilon=(\epsilon(1), \dots, \epsilon(n))\in\{\1,\ast,\prime\}^n$, let $\PB_{2;\epsilon}(n)$ be the set of partitions
 $\pi_f\in \PB_{2}(n)$ such that when $\pi$ is written as
$
\pi =\{\{z_1,w_1\}, \dots, \{z_{n/2},w_{n/2}\}\} 
$  ($z_i<w_i$), then
$$
\e(z_i)=\ast \text{ and }  \e(w_i)= \left\{ \begin{array}{ll}
 \1& \textrm{if $f(\{z_i,w_i\})=1$}\\
\prime & \textrm{if $f(\{z_i,w_i\})=-1$},
\end{array} \right.
$$
for all $1 \leq i \leq n/2$. 

\noindent In order to simplify notation, for $i(1),\dots,i(2n)\in[N]$, $\pi_f\in\PB_{2;\epsilon}(2n)$ such that
 $\ker\underline{i}=\pi$ 
  we will denote 
\begin{align} \label{eq:defcoef}
 \begin{split}\Cof(\pi_f,{\underline{i}}):=&\prod_{\substack{\{z_j,w_j\},\{z_k,w_k\}\in\pi\\ \{z_j,w_j\}\stackrel{\text{cr}}{\sim}\{z_k,w_k\}\\f(\{z_j,w_j\})=1 }} \muh_{\ast,\1}(i(z_k),i(w_j))  \prod_{\substack{\{z_j,w_j\},\{z_k,w_k\}\in\pi\\ \{z_j,w_j\}\stackrel{\text{cr}}{\sim}\{z_k,w_k\}\\f(\{z_j,w_j\})=-1 }} \muh_{\ast,\prime}(i(z_k),i(w_j))\times
  \\  &
\prod_{\substack{ \{z_j,w_j\},\{z_k,w_k\}\in\pi\\  \{z_j,w_j\}\stackrel{\text{nest}}{\sim}\{z_k,w_k\} \\f(\{z_k,w_k\})=1}}\muh_{\ast,\ast}(i(z_j),i(z_k))\muh_{\ast,\1}(i(z_j),i(w_k)) \times
 \\  &    
\prod_{\substack{ \{z_j,w_j\},\{z_k,w_k\}\in\pi\\  \{z_j,w_j\}\stackrel{\text{nest}}{\sim}\{z_k,w_k\}\\f(\{z_k,w_k\})=-1 }}\muh_{\ast,\ast}(i(z_j),i(z_k))\muh_{\ast,\prime}(i(z_j),i(w_k)),
 \end{split} 
\end{align}
where  $\muh_{\cdot,\cdot}(\cdot,\cdot)\in \R$, will be specified in next section, 
%In the proof we also use the notation $\Cof_{\rc}(\{z,w\})$ and $\Cof_{\InNB}(\{z,w\})$ in order to mark in equation \eqref{eq:defcoef}   
crossing generated by  $\{z,w\}$ and the  pairs nested by  $\{z,w\}$, respectively. %($\Cof =\Cof_{\rc}\times \Cof_{\InNB}$).
\section{CLT of type B}  \label{sekcjaCLT}
In this paper we are interested in $\ast$--probability spaces,
since this is the framework which provides us a nice example of random matrices.
%In order to formulate abstract version  CLT of type B we should introduce noncommutative random variables.
\emph{A noncommutative probability space} $(\A,\state)$, if formed by a $\ast$--algebra $\A$, endowed with an antilinear $\ast$--operation and a positive, 
unital linear functional $\state: \A \to \C$, playing the role of expectation.
The  elements $X\in \A$ are called \emph{noncommutative random variables} (see \cite{NS06,VDN92,MS17} for more details). 
 
\begin{Cond} \label{warunek1}
We  
assume that $(\A,\state)$ contains some special sequence of operators $\{\T_i^{\e(i)}\}_{i\in\N}$, $\e(i)\in\{1,\ast,\prime\}$ (marked by bold letters), 
which satisfy the following conditions:
\begin{enumerate}
\item (vanishing means and some second mixed moments) for all $i \in \N$, we have  % and $\e\in\{\1,\prime\}$,
 \begin{align*}&\state(\T_i^\ast) = \state(\T_i)=\state(\T_i^\prime)=0,\\
 &\state(\T_i^*\T_i^*)=\state(\T_i\T_i)=\state(\T_i\T_i^\ast)=\state(\T_i\T_i^\prime)=\state(\T^\prime_i\T_i^\prime)=\state(\T^\prime_i\T_i)=\state(\T^\prime_i\T_i^\ast)=0;
 \end{align*}
\item (uniform bounds) for  $\pi\in\P(n)$, some non-negative real $\varrho_\pi\in\R$ and all $\underline{i}\in \N^n$ such that $\ker \underline{i}=\pi,$  
the following inequality holds
\begin{align*}
\abs{\state\big(\prod_{j=1}^n\T_{i(j)}^{\e(j)}\big)}\leq\varrho_\pi;
\end{align*}
\item (state $\state$  factors over the interval partition) let $$\pi=\{\{1,\dots,k_1\},\{k_1+1,\dots,k_2\},\dots,\{k_{|\pi|-1}+1,\dots,k_{|\pi|}\}\}$$ 
be an interval
partition of $[n]$ and  $\underline{i}\in \N^n$ such that $\ker \underline{i}=\pi,$   then 
\begin{align*}
\state\big(\prod_{j=1}^n\T_{i(j)}^{\e(j)}\big)=\state\big(\T_{i(1)}^{\e(1)}\dots\T_{i(k_1)}^{\e(k_1)}\big)\state\big(\T_{i(k_1+1)}^{\e(k_1+1)}\dots\T_{i(k_2)}^{\e(k_2)}\big)\dots\state\big(\T_{i(k_{|\pi|-1}+1)}^{\e(k_{|\pi|-1}+1)}\dots\T_{i(k_{|\pi|})}^{\e(k_{|\pi|})}\big).
\end{align*}
This condition is in some sense equivalent to independence;
\item (commutation relation) for  $i \neq j$, $i , j\in \N$ and all $\e,\e'\in\{\1,\ast,\prime\}$ our operators satisfy  the relationship
\begin{align*}
&\T_i^\e\T^{\e'}_j=\muh_{\e,\e'}(i,j)\T^{\e'}_j\T_i^{\e},
\end{align*}
with the real-valued coefficients $\muh_{\e,\e'}(i,j)$;
\item  (asymptotic 
existence) for  all  $\pi_f \in \PB_{2,\epsilon}(2n)\neq \emptyset,$ the following limit exists
\begin{align}
 \begin{split}\lim_{N\to \infty }N^{-n}\sum_{\substack{ i(1),\dots,i(2n)\in[N] \\ \ker\underline{i} = \pi} }\left[\Cof(\pi_f,{\underline{i}})\prod_{\substack{\{z,w\} \in \pi\\ f(\{z,w\})=1} }\state(\T^\ast_{i(z)} \T_{i(w)}) \prod_{\substack{\{z,w\} \in \pi\\ f(\{z,w\})=-1}} \state( \T^\ast_{i(z)} \T^\prime_{i(w)})\right] .
 \end{split} \label{eq:limit}
\end{align}
Throughout the paper   $\lambda_{\pi_f}$ denotes the limit above.
\end{enumerate}
\end{Cond}
\begin{remark} 
(1). If  $\PB_{2,\epsilon}(2n)=\emptyset$,  (for example $\e(1)=\dots=\e(2n)=\ast$)  then we understand that  $\lambda_{\pi_f}=0.$
\newline
(2). \noindent We should assume that 
one of the two mixed moments vanishes: $\state(\T_i^*\T_i^\e)=0\text{ or }\state(\T_i^\e\T_i^*)=0 \text{ for } \e\in\{1,\prime\}.$
We chose $\state(\T_i^*\T_i^\e)\neq 0$, because this is compatible with the geometry form in \cite{BEH15}, where we used the right creator.
\newline
\noindent
(3). In order to keep the  positivity requirements, we should assume that the coefficients  $\muh_{\e,\e'}(i,j)$  satisfy some additional assumptions (they cannot take arbitrary values).  For example $$\state(\T_1^\ast\T_2^\ast\T_2\T_1)=\state((\T_2\T_1)^\ast\T_2\T_1)=\muh_{\ast,\1}(2,1)\muh_{\1,\1}(2,1)\state(\T_1^\ast\T_1)\state(\T_2^\ast\T_2)\geq 0.$$
This implies that  $\muh_{\ast,\1}(2,1)\muh_{\1,\1}(2,1)\geq 0$.  It is worth to emphasize that some other relations of this coefficient follow from  
assumption that  $\state$  factors over the interval partition. For example, $\state(\T_1^\ast\T_1\T_2^\ast\T_2)=\state(\T_2^\ast\T_2\T_1^\ast\T_1)$, which 
simply implies that $\muh_{\ast,\ast}(1,2)\muh_{1,\ast}(1,2)\muh_{1,1}(1,2)\muh_{\ast,1}(1,1)=1.$ 
Rather than providing more explicit conditions for the  corresponding  relation we just 
a priori assume that sequence of operators $\{\T_i^{\e(i)}\}_{i\in\N}$ come from a concrete $\ast$--algebra and coefficients above are good.  
\end{remark}
 We now state the main result of this section, which extends the “deterministic formulation” of the noncommutative CLT
of \cite{S92}. 'Deterministic' means that the commutation sequence is now fixed. In the next section we will show that these coefficients may be randomly selected.
\begin{theorem}\label{TwCLT} Let  $\{\T_i,\T_i^{\prime},\T_i^{\ast}\}_{i\in\N}\in\A$ be a sequence of operators, which satisfy Assumption \ref{warunek1}.
Then we have for the sums
$$\Su_N=\frac{\T_{1}+\dots+\T_{N}}{\sqrt{N}},\text{ }\Su^\ast_N=\frac{\T^\ast_{1}+\dots+\T^\ast_{N}}{\sqrt{N}}\text{ and } \Su^\prime_N=\frac{\T^\prime_{1}+\dots+\T^\prime_{N}}{\sqrt{N}},$$
 for all even $r\in \N$
\begin{align}\label{glownerownanie}
&\lim_{N\to \infty}\state(\Su_N^{\e(1)}\dots\Su_N^{\e(r)})=\sum_{\pi_f\in\PB_{2,\epsilon}(r)}\lambda_{\pi_f}.
\end{align}
%whenever the limit on the right side  of \eqref{glownerownanie} exists. 
If $r$  is odd, then the limit above is zero. 
\end{theorem}
\begin{proof}Let us begin to show first that only pair partitions contribute to the corresponding limit. With the notation
%We will denote by $\M$ the subsums over equivalence relation established by $\ker\underline{i}=\pi $
%Let us fix $r\in \N$  and  calculate
\begin{align*}\M&=\frac{1}{N^{r/2}}\sum_{\substack{ i(1),\dots,i(r)\in[N] \\ \ker\underline{i}=\pi }}\state(\T_{i(1)}^{\e(1)}\dots\T_{i(r)}^{\e(r)}),
\intertext{we have}
\state(\Su_N^{\e(1)}\dots\Su_N^{\e(r)})&=\frac{1}{N^{r/2}}\sum_{ i(1),\dots,i(r)\in[N]}\state(\T_{i(1)}^{\e(1)}\dots\T_{i(r)}^{\e(r)})
%\intertext{ d}
=\sum_{\pi \in \P(r)}\M.
\end{align*}
First, we will show that partitions with singletons do not contribute to  $\M.$  Consider a partition  $\pi$  with a singleton. 
Then we can rewrite expression $\T_{i(1)}^{\e(1)}\dots\T_{i(r)}^{\e(r)}$  for every  $\ker\underline{i}= \pi$  (via
the commutation relations) into a form associated with interval partition  $\overline{\pi}$. 
In the new situation state  $\state$  factors the blocks in $\overline{\pi}$  and  $\state(\T_i^\ast) = \state(\T_i)=\state(\T_i^\prime)=0$,  
thus we get  $\M=0.$
\newline
\newline
Thus only such $\pi$ partitions contribute which have no singletons. Note that this implies that we can restrict our sums over $\pi\in\P(r)$  
to  $|\pi|\leq \floor{\frac{r}{2}}$.  Recalling that, by the assumption of the existence of uniform bounds for the moments 
%$$\state\big(\prod_{j=1}^k\T_{i(j)}^{\e(j)}\big)\leq\varrho_k$$
and equation $$\sum_{\substack{ i(1),\dots,i(r)\in[N] \\ \ker\underline{i}= \pi} }1={N \choose |\pi|}|\pi|!,$$  we can estimate  
$|\M|\leq \frac{1}{N^{r/2}} {N \choose |\pi|}|\pi|!\varrho_\pi$ for some $\varrho_\pi\in \R^+$. Finally, we see that 
\begin{align}
\abs{\state(\Su_N^{\e(1)}\dots\Su_N^{\e(r)})}\leq\sum_{\pi\in\P(r)} \frac{1}{N^{r/2}} {N \choose |\pi|}|\pi|!\varrho_\pi.
\label{eq:ograniczenie}
\end{align}
Note that  $\lim_{N\to\infty} \frac{1}{N^{r/2}} {N \choose |\pi|}|\pi|!=0$ for every $|\pi|< \frac{r}{2}$
and  for such $\pi$  limit of $|\state(\Su_N^{\e(1)}\dots\Su_N^{\e(r)})|$, when $N\to \infty$, is equal to zero, because the sum in inequality \eqref{eq:ograniczenie} above is taken over a fixed $r$. This
means  that $|\pi|={r}/{2}$, thus $r$ must be even and $\pi$ has to be a pair partition. 
\newline
\newline
Let $r=2n$ and let us fix the partition $\pi=\{\{z_1,w_1\},\dots,\{z_{n},w_{n}\}\}\in\P_2(2n)$ designed through the multiindex $\underline{i}$ by $\ker\underline{i}=\pi$.
At the outset, recall that expression $\state(\T_{i(1)}^{\e(1)}\dots\T_{i(2n)}^{\e(2n)})$ (for $\ker\underline{i}=\pi$) has a possible non-zero value 
    if and only if $(\e(z_j),\e(w_j))\in\{(\ast,\1),(\ast,\prime)\}$ for all $j\in[n]$ (which is a simple consequence of Assumption \ref{warunek1}(a)). 
    This means that  $\pi_f\in\PB_{2,\e}(2n)$,
 where the coloring $f$ is determined by pairs  $(\ast,\1)\text{ and }(\ast,\prime)$.
 Now we will show that for the above-established partition  $\pi_f$ and all  indexes $\underline{i}$ such that $\ker\underline{i}=\pi$, commutation rules lead us to  
\begin{align}
\T_{i(1)}^{\e(1)}\dots\T_{i(2n)}^{\e(2n)}=\Cof(\pi_f,{\underline{i}})\T_{i(z_1)}^{\ast}\T_{i(w_1)}^{\e(w_1)}\dots\T_{i(z_{n})}^{\ast}\T_{i(w_{n})}^{\e(w_{n})}, \label{eq:induction}
\end{align}
where $\e(w_j)\in\{\1,\prime\}$, $j\in[n]$. 
The proof is given by induction. If $n=1$, then we have two pairs in $\PB_{2,\e}(2)$  i.e. $f(\{1,2\})=1$ and $f(\{1,2\})=-1$, which corresponds to
 $(\e(1),\e(2))=(\ast,\1)$ and  $(\e(1),\e(2))=(\ast,\prime)$, respectively. In this two situation  $\Cof(\pi_f,{\underline{i}})=1$ and hence the formula is true.
Suppose that the formula is true for $n-1$. %,.  then %for any $\epsilon\in\{1,\ast,\prime\}^{k}$, we get 
We assume, that $\pi$ has 
\begin{itemize}
\item pairs $U_1,\dots, U_s$, which
 are crossing the pair $\{z_{n},w_{n}\}$,
%\item arcs $V_1,\dots, V_t$ with color $-1$ to the right of $\ith^{\rm th}$ in the strict sense,
\item pairs $W_1,\dots, W_t$, which are covered by $\{z_{n},w_{n}\}$.
\end{itemize}
In the proof we also use the notations  $\widetilde{\pi}=\pi\setminus\{\{z_{n},w_{n}\}\}$, $\widetilde{\underline{i}}=(i(z_1),\dots ,\check{i}(z_n),\dots ,\check{i}(w_n))$  
(superscript $\check{i}$  indicates that  $i$  has been deleted from the multiindex).
\newline
\newline
\noindent First, let us recall that pair $\{z_{n},w_{n}\}$ is the most right in $\pi$, namely $w_n=2n$ (which means that pairs $U_i$ cross them from the left side 
-- see Figure \ref{fig:FiguraExemple1}(a)). The general strategy of the proof is to use the commutation relation and shift the operator $\T^\ast_{i(z_n)}$  to the 
right until we "meet" the element $\T^{\e(w_n)}_{i(w_n)}$ (on the position $2n-1$). During this process 
we have %a new coefficient $\Cof(\widetilde{\pi},\widetilde{\underline{i}})$ is obtained  in a
 two situations.
\newline
\newline
\noindent Situation 1. When we move the operator $\T^\ast_{i(z_n)}$ to the right, we find the crossing  pair  $U_j=\{z_j,w_j\}$. In this case 
 we use the commutation relation between  $\T^\ast_{i(z_n)}$ and $\T^{\e(w_j)}_{i(w_j)}$. The corresponding contribution to coefficient is therefore given by $\muh_{\ast,\e(w_j)}(i(z_n),i(w_j))$.
This operation graphically corresponds to exchange indices $z_{n}$ and $w_j$, to yield a new type-B pair partition where $\{z_j,w_j\}$ to the left of $\{z_{n},w_{n}\}$ 
in the strict sense -- see Figure \ref{fig:FiguraExemple1}(a). Summarizing this, we get 
\begin{align*}
%\frac{\Cof_\rc(\pi,{\underline{i}})}{\Cof_\rc(\widetilde{\pi},\widetilde{\underline{i}})}&=\\&
\rc(\{z_n,w_n\})&=
\prod_{\substack{\{z_j,w_j\}\in\{U_1,\dots, U_s\}\\ \{z_j,w_j\}\stackrel{\text{cr}}{\sim}\{z_n,w_n\} }} \muh_{\ast,\e(w_j)}(i(z_n),i(w_j))
\\&=\prod_{\substack{\{z_j,w_j\}\in\{U_1,\dots, U_s\}\\ \{z_j,w_j\}\stackrel{\text{cr}}{\sim}\{z_n,w_n\}\\f(\{z_j,w_j\})=1 }} \muh_{\ast,\1}(i(z_n),i(w_j))  \prod_{\substack{\{z_j,w_j\}\in\{U_1,\dots, U_s\}\\ \{z_j,w_j\}\stackrel{\text{cr}}{\sim}\{z_n,w_n\}\\f(\{z_j,w_j\})=-1 }} \muh_{\ast,\prime}(i(z_n),i(w_j)).
\end{align*}
\noindent Situation 2. When we shift the operator $\T^\ast_{i(z_n)}$  we encounter the nest  pair  $W_j=\{z_j,w_j\}$. Then, by using commutation relation two 
new terms appear. In the first action between  $\T^\ast_{i(z_n)}$  and $\T^\ast_{i(z_j)}$ we obtain the expression $\muh_{\ast,\ast}(i(z_n),i(z_j))$ -- see 
Figure \ref{fig:FiguraExemple1}(b). Next in the relationships between  $\T^\ast_{i(z_n)}$  and $\T^{\e(w_j)}_{i(w_j)}$ the second coefficient appears  $\muh_{\ast,\e(w_j)}(i(z_n),i(w_j))$ -- see Figure \ref{fig:FiguraExemple1}(c). Similarly as in Situation 1 this step can be illustrated to get a new type-B set partition, where the $\{z_{n},w_{n}\}$, 
have interval form. Altogether, we get
\begin{align*}&\displaystyle%\frac{\Cof_{\InNB}(\pi_f,{\underline{i}})}{\Cof_\InNB(\widetilde{\pi}_f\widetilde{\underline{i}})}=
{\InNB}(\{z_n,w_n\})=
\prod_{\substack{ \{z_j,w_j\}\in\{W_1,\dots,W_t\}\\  (z_n,w_n)\stackrel{\text{nest}}{\sim}\{z_j,w_j\} }}\muh_{\ast,\ast}(i(z_n),i(z_j))\muh_{\ast,\e(w_j)}(i(z_n),i(w_j))
=\\&\prod_{\substack{ \{z_j,w_j\}\in\{W_1,\dots,W_t\}\\  (z_n,w_n)\stackrel{\text{nest}}{\sim}\{z_j,w_j\} \\f(\{z_j,w_j\})=1}}\muh_{\ast,\ast}(i(z_n),i(z_j))\muh_{\ast,\1}(i(z_n),i(w_k)) 
\prod_{\substack{ \{z_j,w_j\}\in\{W_1,\dots,W_t\}\\  (z_n,w_n)\stackrel{\text{nest}}{\sim}\{z_j,w_j\}\\f(\{z_j,w_j\})=-1 }}\muh_{\ast,\ast}(i(z_n),i(z_j))\muh_{\ast,\prime}(i(z_n),i(w_j)).
\end{align*} 
  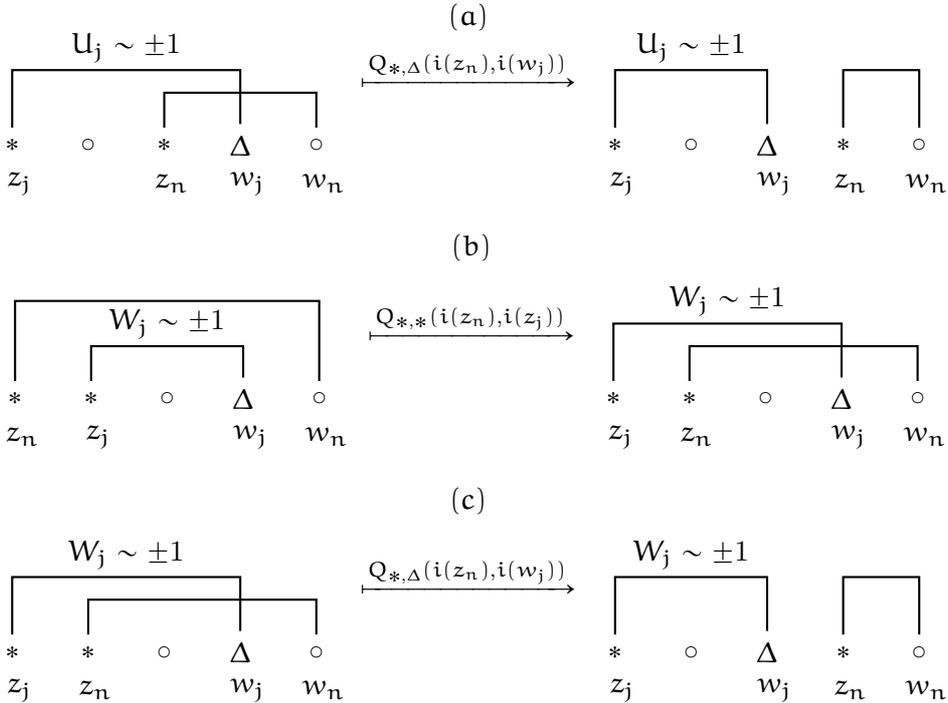
\begin{figure}[h]
\begin{center}
   \begin{tikzpicture}[thick,font=\small,scale=1]
     \path 
           (7,1.7) node[] (a) {$(a)$}
           (1,0) node[] (b) {$\ast$}
           (2,0) node[] (c) {$\circ$}
           (3,0) node[] (d) {$\ast$}
           (4,0) node[] (e) {$\Krop$}
           (5,0) node[] (f) {$\circ$}
          % (6,0) node[] (g) {$\circ$};
            (1.1,-0.5) node[] (w1) {$z_j$}
            (4.1,-0.5) node[] (w1) {$w_j$}
            (3.1,-0.5) node[] (w1) {$z_n$}
            (5.1,-0.5) node[] (w1) {$w_n$}
            (2.5,1.3) node (fdf) {$U_j \sim \pm1$}
            (7,1) node (fdf) {$\xmapsto{\muh_{\ast,\Krop}(i(z_n),i(w_j))}$ };
            %(4.5,1) node (fdf) {$ \pm1$}
           %(6,0) node[] (h) {$\longrightarrow$}
     %\draw (a) -- +(0,0.75) -| (d);
     \draw (b) -- +(0,1) -| (e);
     \draw (d) -- +(0,0.7) -| (f);
     
   \end{tikzpicture}
    \begin{tikzpicture}[thick,font=\small,scale=1]
    \path 
          % (0,0) node[] (a) {$\circ$}
           (1,0) node[] (b) {$\ast$}
           (2,0) node[] (c) {$\circ$}
           (3,0) node[] (d) {$\Krop$}
           (4,0) node[] (e) {$\ast$}
         (1.1,-0.5) node[] (w1) {$z_j$}
            (4.1,-0.5) node[] (w1) {$z_n$}
            (3.1,-0.5) node[] (w1) {$w_j$}
            (5.1,-0.5) node[] (w1) {$w_n$}
              (2,1.3) node (fdf) {$U_j \sim \pm1$}
           % (4.5,1.3) node (fdf) {$ $}
           (5,0) node[] (f) {$\circ$};
     %\draw (a) -- +(0,0.75) -| (d);
     \draw (b) -- +(0,1) -| (d);
     \draw (e) -- +(0,1) -| (f);
   \end{tikzpicture}
   \end{center}  
   %\hspace{0cm}
   \begin{center}
 \begin{tikzpicture}[thick,font=\small,scale=1]
     \path 
          (7,2) node[] (a) {$(b)$}
           (1,0) node[] (b) {$\ast$}
           (2,0) node[] (d) {$\ast$}
           (3,0) node[] (c) {$\circ$}
           (4,0) node[] (e) {$\Krop$}
           (5,0) node[] (f) {$\circ$}
          % (6,0) node[] (g) {$\circ$};
            (1.1,-0.5) node[] (w1) {$z_n$}
            (4.1,-0.5) node[] (w1) {$w_j$}
            (2.1,-0.5) node[] (w1) {$z_j$}
            (5.1,-0.5) node[] (w1) {$w_n$}
            %(3.5,1.8) node (fdf) {$\pm 1$}
            (7,1) node (fdf) {$\xmapsto{\muh_{\ast,\ast}(i(z_n),i(z_j))}$ }
            (3,1) node (fdf) {$W_j\sim\pm 1$};
           %(6,0) node[] (h) {$\longrightarrow$}
     %\draw (a) -- +(0,0.75) -| (d);
     \draw (b) -- +(0,1.3) -| (f);
     \draw (d) -- +(0,0.7) -| (e);
     
   \end{tikzpicture}
      \begin{tikzpicture}[thick,font=\small,scale=1]
     \path 
         %  (7,2) node[] (a) {$(a)$}
           (1,0) node[] (b) {$\ast$}
           (2,0) node[] (c) {$\ast$}
           (3,0) node[] (d) {$\circ$}
           (4,0) node[] (e) {$\Krop$}
           (5,0) node[] (f) {$\circ$}
          % (6,0) node[] (g) {$\circ$};
            (1.1,-0.5) node[] (w1) {$z_j$}
            (4.1,-0.5) node[] (w1) {$w_j$}
            (2.1,-0.5) node[] (w1) {$z_n$}
            (5.1,-0.5) node[] (w1) {$w_n$}
            (2.5,1.3) node (fdf) {$W_j\sim \pm1$};
          %  (7,1) node (fdf) {$\xmapsto{\muh_{\ast,\Krop}(i(z_n),i(w_j))}$ };
            %(4.5,1) node (fdf) {$ \pm1$}
           %(6,0) node[] (h) {$\longrightarrow$}
     %\draw (a) -- +(0,0.75) -| (d);
     \draw (b) -- +(0,1) -| (e);
     \draw (c) -- +(0,0.7) -| (f);
   \end{tikzpicture}
   \end{center}
 %  (8,1) node (fdf) {$\xmapsto{\muh_{\ast,\ast}(i(z_n),i(z_j))\muh_{\ast,\Krop}(i(z_n),i(w_j))}$ }
   \begin{center}
       \begin{tikzpicture}[thick,font=\small,scale=1]
     \path 
           (7,2) node[] (a) {$(c)$}
           (1,0) node[] (b) {$\ast$}
           (2,0) node[] (c) {$\ast$}
           (3,0) node[] (d) {$\circ$}
           (4,0) node[] (e) {$\Krop$}
           (5,0) node[] (f) {$\circ$}
          % (6,0) node[] (g) {$\circ$};
            (1.1,-0.5) node[] (w1) {$z_j$}
            (4.1,-0.5) node[] (w1) {$w_j$}
            (2.1,-0.5) node[] (w1) {$z_n$}
            (5.1,-0.5) node[] (w1) {$w_n$}
            (2.5,1.3) node (fdf) {$W_j\sim \pm1$}
            (7,1) node (fdf) {$\xmapsto{\muh_{\ast,\Krop}(i(z_n),i(w_j))}$ };
            %(4.5,1) node (fdf) {$ \pm1$}
           %(6,0) node[] (h) {$\longrightarrow$}
     %\draw (a) -- +(0,0.75) -| (d);
     \draw (b) -- +(0,1) -| (e);
     \draw (c) -- +(0,0.7) -| (f);

   \end{tikzpicture}
    \begin{tikzpicture}[thick,font=\small,scale=1]
    \path 
           %(0,0) node[] (a) {$\circ$}
           (1,0) node[] (b) {$\ast$}
           (2,0) node[] (c) {$\circ$}
           (3,0) node[] (d) {$\Krop$}
           (4,0) node[] (e) {$\ast$}
         (1.1,-0.5) node[] (w1) {$z_j$}
            (4.1,-0.5) node[] (w1) {$z_n$}
            (3.1,-0.5) node[] (w1) {$w_j$}
            (5.1,-0.5) node[] (w1) {$w_n$}
              (2,1.3) node (fdf) {$W_j\sim \pm1$}
            %(4.5,1.3) node (fdf) {$ \pm1$}
           (5,0) node[] (f) {$\circ$};
     %\draw (a) -- +(0,0.75) -| (d);
     \draw (b) -- +(0,1) -| (d);
     \draw (e) -- +(0,1) -| (f);
   \end{tikzpicture}
   \end{center}
\caption{The visualization of bringing a crossing and nesting partition into the interval form (in the induction step, where  $\Krop= \e(w_j)$, and the notation
	 $\sim \pm1$ means that a pair has color $\pm1$). 
}
\label{fig:FiguraExemple1}
\end{figure}
\newline
\noindent   
Finally, we can use induction with respect to $\widetilde{\pi}\in\PB_{2,\e}(2n-2)$ and $\widetilde{\underline{i}}$, because in the new situation the shifting 
operator $\T^\ast_{i(z_n)}$  together with $\T^{\e(w_n)}_{i(w_n)}$ correspond to the pair $\{z_n,w_n\}$ that does not affect crossing and nesting in partition $\widetilde{\pi}$, i.e. 
\begin{align*} 
\T_{i(1)}^{\e(1)}\dots\T_{i(2n)}^{\e(2n)}=\rc(\{z_n,w_n\})\InNB(\{z_n,w_n\})\underbracket{\T_{i(z_1)}^{\ast} \dots \check{\T}_{i(z_n)}^{\ast}\dots \T_{i(w_{n-1})}^{\e(w_{n-1})}}_{\text{induction} }\T_{i(z_{n})}^{\ast}\T_{i(w_{n})}^{\e(w_{n})}.
\end{align*}

\noindent  Now we can apply the factorization assumption to equation \eqref{eq:induction}, which yields  
\begin{align*}
\state\big(\T_{i(1)}^{\e(1)}\dots\T_{i(2n)}^{\e(2n)}\big)&=\Cof(\pi_f,{\underline{i}})\state\big(\T_{i(z_1)}^{\ast}\T_{i(w_1)}^{\e(w_1)}\big)\dots\state\big(\T_{i(z_{n})}^{\ast}\T_{i(w_{n})}^{\e(w_{n})}\big)
\\&=\Cof(\pi_f,{\underline{i}})\prod_{\substack{\{z,w\} \in \pi\\ f(\{z,w\})=1} }\state(\T^\ast_{i(z)} \T_{i(w)}) \prod_{\substack{\{z,w\} \in \pi\\ f(\{z,w\})=-1}} \state( \T^\ast_{i(z)} \T^\prime_{i(w)}).
\end{align*}
In the end, we sum over all possible choices of indices $i(1),\dots,i(2n)\in[N]$ such that $\ker\underline{i}=\pi$, which by existence of limit \eqref{eq:limit} 
 yields \eqref{glownerownanie} and completes the proof.
\end{proof}
\section{The Almost Sure Convergence}
Now we shall describe a concrete $\ast$--algebra and construct a family of random matrices (whose entries are classical random
variables), which fulfills  Assumption \ref{warunek1}, where we focus on the existence of limit \eqref{eq:limit}.  As a corollary, we conclude that these 
random matrices asymptotically behave like the Gaussian operator of type B. 

We assume that the entries of random matrices are the three kinds of random, independent variables $\{\muh(i,j)\}_{i\neq j \in \N}$, $\{\mutylda(i,j)\}_{  i\neq j \in \N }$ 
and $\{\Psi_i\}_{i\in \N}$, and we denote the probability measure on the probability space of these random structures by  $\Pro$. 
 Families  $\{\muh(i,j)\}_{i< j }$ and  $\{\mutylda(i,j)\}_{  i\neq j  }$  are drawn from a
collection of independent and identically distributed random variables (inside of each family), and $\muh(i,j)=\muh(j,i)$ for $i>j$. What is more, we assume that 
random variables $\{\Psi_i\}_{i\in \N}$ are independent. In order to get asymptotic results, it is necessary to impose some additional restrictions:
\begin{enumerate}
\item  $\muh(i,j)$  have a compact
support $D_1\subset \R$ which is separated from zero i.e $D_1\subsetneq \R\setminus (\epsilon,\epsilon)$, for some $\epsilon >0;$ 
\item $\mutylda(i,j)$ have a compact 
support $D_2\subset \R$ which is separated from zero;
\item $\Psi_i$  have a compact 
support contained in the set $D_3\subset \R$ and  $\sum_{i=1}^\infty Var(\Psi_i)<\infty.$
\end{enumerate}

 \begin{remark}
  Assumption $\sum_{i=1}^\infty Var(\Psi_i)<\infty$ is purely technical and use it in order to estimate a variance (this condition can be weakened, but 
  then it is not elegant). Random variables  $\{\Psi_i\}_{i\in \N}$  are associated with the parameter  $\alpha$. Our typical example is just a constant 
  function $\Psi_i=\alpha,$ for $i\in \N.$ Actually, we could have assumed in advance that the parameter $\alpha$ is deterministic in order to facilitate 
  some estimations, but our primary motivation was to make the model completely random. 
    \end{remark}
Let $M_2(\R)$ denote the algebra of $2\times 2$ real matrices. In the construction below we use special notation matrices $\{\sigma_x,\gamma_x\}_{x\in \R}\in M_2(\R)$ 
given by   
     \begin{align*}
      \sigma_x= 
      \begin{bmatrix}
        1 & 0  \\ 
          0&  x 
      \end{bmatrix}, 
%\qquad      
 %       \stylda_{y}= 
  %    \begin{bmatrix}
   %     1 & 0  \\ 
    %      0&  y 
     % \end{bmatrix}, 
\qquad      
         \gamma_{x}= 
      \begin{bmatrix}
        0 &  0\\ 
         x   &  0 
      \end{bmatrix}.   
       % \zeta_1= 
      %\begin{bmatrix}
       % 1 &  0\\ 
        % 0   &  0 
      %\end{bmatrix}.
      %\qquad      
       % \zeta_2= 
      %\begin{bmatrix}
       % 0 &  0\\ 
       %  0   &  1 
      %\end{bmatrix}
    \end{align*}
\subsection{Random Matrix Model }
We define the $\ast$--representation as $\A_N:=M_2(\R)^{\otimes N}$, where the $\ast$ operation is the conjugate transpose. Let the elements
 $\T_{N,i},\T^\prime_{N,i}\in \A_N$ be given by
    \begin{align*} \T_{N,i}&=\sigma_{\muh({1,i})}\otimes \dots \otimes \sigma_{\muh({i-1,i})}\otimes \gamma_1 \otimes 
    \underbracket{ \sigma_1\otimes \dots \otimes \sigma_{1}}_{N-i \text{ times}},
      \\
      \T^\prime_{N,i}&=\sigma_{\mutylda({1,i})}\otimes \dots \otimes \sigma_{\mutylda({i-1,i})}\otimes \gamma_{\Psi_i} \otimes 
    \underbracket{ \sigma_{\muh(i+1,i)\mutylda(i+1,i)}\otimes \dots \otimes \sigma_{\muh(N,i)\mutylda(N,i)}}_{N-i \text{ times}}.
    \end{align*} 
    \begin{remark}
    In literature, the element  $\T_{N,i}$ is called Jordan-Wigner-transform and is well known to produce anti-commuting variables 
    (e.g. \cite{Bi97,EE93,S92}). The second matrix $\T^\prime_{N,i}$ is new and  might be called full or extended  Jordan-Wigner transform, 
    because there is no identity part, i.e. $ \sigma_1\otimes \dots \otimes \sigma_{1}.$
    \end{remark}
Note that the operator $\T_{N,i}^\ast$ is obtained by the transpose of $\T_{N,i}$.
Furthermore, let $\state_N:\A_N\to\C$ be the positive map $\state_N(a_1\otimes\dots \otimes a_N)=\langle e_1 a_1,e_1\rangle \dots \langle e_1 a_N,e_1\rangle$, 
where $a_1,\dots, a_N\in M_2(\R)$, $\langle \cdot,\cdot\rangle$  is the usual inner product on $\R^2$ and $e_1=
(1,0)$ is an element of the standard basis.   
We claim that the elements above satisfy Assumptions \ref{warunek1}(a)-(d). 
Indeed, restrictions $(a)$ and $(c)$ follow directly from definition. 
%and state  $\state_n$ factors  the interval partition. 
Note that $\sigma_x\sigma_y=\sigma_y\sigma_x$ and $\gamma_y^\ast\sigma_x=x\sigma_x\gamma_y^\ast,$ which by elementary manipulations on tensor products 
implies that Assumption \ref{warunek1}(d) is satisfied. The coefficients $\muh_{\e,\e'}(i,j)$  are  equal to 
 \begin{align}
 \label{eq:relationimprtantinproof1}&\muh_{\ast,\ast}(i,j)=\muh(i,j)\qquad \muh_{\ast,\1}(i,j)= \muh(i,j)\quad \muh_{\ast,\prime}(i,j)= \mutylda(i,j)&&\text{ for } i<j,
 \\ \label{eq:relationimprtantinproof2}&\muh_{\ast,\ast}(i,j)=\muh^{-1}(i,j)\quad \muh_{\ast,\1}(i,j)= \muh(i,j)\quad \muh_{\ast,\prime}(i,j)=\underbracket{\muh(j,i)\muh(i,j)}_{\muh^2(i,j)}\mutylda(i,j) &&\text{ for } i>j,
\end{align}
where we write directly only these coefficients which are important in the proof of Proposition \ref{prop:stochinterpolacja} below (some of remaining 
relations now follow by taking adjoints of equations \eqref{eq:relationimprtantinproof1} and \eqref{eq:relationimprtantinproof2}).  By the same token, 
we also get that $\muh_{\e,\e'}(i,j)$, where $\e,\e'\in\{1,\ast,\prime\}$ can be expressed as $\muh^a(i,j)\muh^b(j,i)\mutylda^c(i,j)$ for some $a,b,c\in\{-1,0,1\}.$ 
Thus, by establishing compact supports which are separated from zero, we can estimate
 $$|\muh_{\e,\e'}(i,j)|\leq \left[\max\big(\sup_{x\in D_1}|x|,\sup_{x\in D_1}(1/|x|)\big)\right]^2\max\big(\sup_{x\in D_2}|x|,\sup_{x\in D_2}(1/|x|)\big)=\K.$$ 
%with $\K\geq 1.$
Furthermore, observe that for $\ker \underline{i}=\pi$
\begin{align*}
\left|\state_N\Big(\prod_{j=1}^n\T_{N,i(j)}^{\e(j)}\Big)\right|\leq \K^{Function(\pi)}\left(\max(1,\sup_{x\in D_3}|x|)\right)^{\#\pi}.
\end{align*}
Indeed we can use finitely many steps in order to transform  $\prod_{j=1}^n\T_{N,i(j)}^{\e(j)}$ into a form ${C(\pi)}\times \widetilde{\T}_{N,i(j)}^{\e(j)}$  
associated with the interval  partition $\widetilde{\pi}$, where $C(\pi)$ is a product of some coefficients $\muh_{\e,\e'}$, which are bounded by $\K$. 
 Now if $\widetilde{\pi}$ has the interval form, then the state $\state_N(\prod_{j=1}^n \widetilde{\T}_{N,i(j)}^{\e(j)})$ may take three possible values:
  $0,1$ or $\prod_{j\in S}\Psi_{i(j)}$ for some 
$S\subset [n]$, which finally explain the inequality above.% (but is not the best constant). 
\begin{remark}
(1). In order to improve the clarity of writing in next subsection we skip index $N$ in $\state$ and operators $\T.$

\noindent (2). The state $\state$ takes random values, so assumption on a compact 
support is necessary in order to be able to control for higher-order moments.  

\noindent (3). It is worth to mention that operator $ \T_{N,i} + \T_{N,i}^\ast+\T_{N,i}^\prime$ is not self-adjoint, and becomes self-adjoint only in the 
trivial case, for
$\Psi_i=0$ for all $i\in \N$. 
This element plays a crucial role in asymptotic realization of Gaussian operator of type B. 
 % This is not a problem because 
\end{remark}
\subsection{Stochastic interpolation}
At this point, we have a natural candidate of operators for Theorem \ref{TwCLT}, namely they have good properties, but still we are not sure whether the 
limit \eqref{eq:limit} exists. The following proposition explains this problem with almost sure convergence, i.e. we will show that if we put 
some restrictions on first moments, then we can easily describe the desired limit.
\begin{Prop} \label{prop:stochinterpolacja}
 Fix $q,\alpha\in(-1,1)$ and assume that  %and let $\{\muh(i,j)\}_{1\leq i<j}$,  $\{\mutylda(i,j)\}_{1\leq i<j}$ and  
\begin{align*}
&\E\big(\muh(i,j)\big)=q,\quad \E\big(\muh^2(i,j)\big)=1,\quad \E\big(\mutylda(i,j)\big)=q\quad\text{and}\quad  \E(\Psi_i)=\alpha.\end{align*}
Then for every $\pi_f\in \PB_{2,\e}(2n)$ %and $\underline{\e}\in\{\ast,1,\prime\}$  
the limit \eqref{eq:limit} exists (created from the entries of random matrices $\T^\ast_i,\T_i,\T^\prime_i$) and equals 
$$\lambda_{\pi_f}=\alpha^{\NB(\pi_f)}q^{\rc(\pi)+2 \InNB(\pi_f)}.$$ 
\label{prop:stochastycznainteroplacja}
\end{Prop}
\begin{proof}
For a fixed $\pi_f\in\PB_{2,\e}(2n)$, we consider the classical
random variable  as below
\begin{align} \label{eq:klasycznazmienna}
\Y=\lim_{N\to \infty }N^{-n}\sum_{\substack{ i(1),\dots,i(2n)\in[N] \\ \ker\underline{i} = \pi} }\left[\Cof(\pi_f,{\underline{i}})\prod_{\substack{\{z,w\} \in \pi\\ f(\{z,w\})=1} }\state(\T^\ast_{i(z)} \T_{i(w)}) \prod_{\substack{\{z,w\} \in \pi\\ f(\{z,w\})=-1}} \state( \T^\ast_{i(z)} \T^\prime_{i(w)})\right].
\end{align}
We claim that $\lim_{N\to \infty}\Y=\lambda_{\pi_f}.$
By $\E(\Y)$ we denote the expectation of random variable $\Y$ with respect to $\Pro$. 
The first goal is to compute $\E(\Y)$, namely we will show that expected value in square bracket of equation \eqref{eq:klasycznazmienna} does not depend 
on index $\underline{i}.$  By the independence assumption on classical random variable (established by $\ker\underline{i} = \pi$) it suffices to evaluate 
the expected value of a corresponding term separately  
for each crossing and nesting partition (contained in $\Cof(\pi_f,{\underline{i}})$)  and suitable coloring from equation \eqref{eq:klasycznazmienna}  
(related to $\state$). 
\newline
\noindent Case 1. We compute expectation of a part  associated  with $\state$.
 Let us first observe that 
    $$
    \state( \T^\ast_{i} \T^\epsilon_{i})= \left\{ \begin{array}{ll}
1 & \textrm{if $\epsilon=\1$}\\
\Psi_{i} & \textrm{if $\epsilon=\prime $},
\end{array} \right.%=\muh_{\ast,1}^{j,i}\TA_{n,j}^\ast\T_{n,i} \text{ }
    $$
    which immediately implies that for $\ker\underline{i} = \pi$, we have
    $$\E\Big(\prod_{\substack{\{z,w\} \in \pi\\ f(\{z,w\})=1} }\state(\T^\ast_{i(z)} \T_{i(w)}) \prod_{\substack{\{z,w\} \in \pi\\ f(\{z,w\})=-1}} \state( \T^\ast_{i(z)} \T^\prime_{i(w)})\Big)=\alpha^{\NB(\pi_f)}.$$
\noindent Case 2. Now we move to the crossing partition, then we can compute the expected value directly form relation \eqref{eq:relationimprtantinproof1}  
and \eqref{eq:relationimprtantinproof2}, namely 
\begin{align*} 
&\E \big(\muh_{\ast,\Krop}(i(z_k),i(w_j))\big)=\left\{ \begin{array}{llll}
 \E\big(\muh(i(z_k),i(w_j))\big)=q & \textrm{if $i(z_k)<i(w_j),\text{ }\Krop=\1$}\\
\E\big(\muh(i(z_k),i(w_j))\big)=q & \textrm{if $i(z_k)>i(w_j),\text{ }\Krop=\1$}
\\
\E\big(\mutylda(i(z_k),i(w_j))\big)=q & \textrm{if $i(z_k)<i(w_j),\text{ }\Krop=\prime$}
\\
%\E\big(\muh^2(i(z_k),i(w_j))\mutylda(i(z_k),i(w_j))\big)\\=
\E\big(\muh^2(i(z_k),i(w_j))\big)\E\big(\mutylda(i(z_k),i(w_j))\big)=q & \textrm{if $i(z_k)>i(w_j),\text{ }\Krop=\prime$}.
\end{array} \right.%=\muh_{\ast,1}^{j,i}\TA_{n,j}^\ast\T_{n,i}.
\end{align*}
Summarizing,  we can say that on average (independently from coloring) each crossing contributes to $q$.
  
\noindent  Case 3. The last situation is the nesting partition. In this case we should evaluate  
\begin{align*}
&\L=\E\big(\muh_{\ast,\ast}(i(z_j),i(z_k))\muh_{\ast,\Krop}(i(z_j),i(w_k))\big)\text{ for } \Krop\in\{\1,\prime\}.
\intertext{To do this, let us divide situation into two steps $(a)$ and $(b)$ and similarly as previously use equations \eqref{eq:relationimprtantinproof1}  
	and \eqref{eq:relationimprtantinproof2}.
 In step $(a)$ we assume that $\Krop=\1$ -- this means that the pair $\{z_k,w_k\}$ has color $1$. If $i(z_j)=i(w_j)<i(z_k)=i(w_k)$, then }
 &\L=\E\big(\muh^2(i(z_j),i(w_k))\big)=1.
\intertext{Otherwise if $i(z_j)=i(w_j)>i(z_k)=i(w_k)$, then expectations are given as}
 &\L=\E\big(\muh^{-1}(i(z_j),i(w_k))\muh(i(z_j),i(w_k))\big)=1.
   \intertext{In the second step  $(b)$  $\Krop=\prime$ (= pair $\{z_k,w_k\}$ has color $-1$). If $i(z_j)=i(w_j)<i(z_k)=i(w_k)$, then }
 &\L=\E\big(\muh(i(z_j),i(w_k))\big)\E\big(\mutylda(i(z_j),i(w_k))\big)=q^2.
\intertext{On the other hand, when $i(z_j)=i(w_j)>i(z_k)=i(w_k)$, then  }
 &\L=\E\big(\muh^{-1}(i(z_j),i(w_k))\muh^2(i(z_j),i(w_k))\mutylda(i(z_j),i(w_k))\big)=q^2.
 %\\&= \E\big(\muh(i(z_j),i(w_k))\big)\E\big(\mutylda(i(z_j),i(w_k))\big). 
\end{align*}
Finally, each nesting contributes a factor $1$ or $q^2$ if covered pair has color $1$ or $-1$, respectively, and we have 
 \begin{align*}
 \begin{split}\E(\Y)=N^{-n}\sum_{\substack{ i(1),\dots,i(2n)\in[N] \\ \ker\underline{i} = \pi} }\alpha^{\NB(\pi_f)}q^{\rc(\pi)+2 \InNB(\pi_f)}= \alpha^{\NB(\pi_f)}q^{\rc(\pi)+2 \InNB(\pi_f)}N^{-n}{N\choose n}n!.
 \end{split}
\end{align*}
Thus  $\lim_{N\to \infty}\E(\Y)=\alpha^{\NB(\pi_f)}q^{\rc(\pi)+2 \InNB(\pi_f)}.$ %The rest of the proof is analogous to that in \cite{S92}.
Now it remains  to show that  $\lim_{N\to \infty}\Y=\E(\Y)$  (in the almost surely sense). For every  $\eta>0$,  we have 
\begin{align*}
\Pro\Big(\bigcup_{M\geq N} \{|\YM-\E(\YM)|>\eta\}\Big)&\leq \sum_{M\geq N} \Pro\left(\{|\YM-\E(\YM)|>\eta\}\right)\\&\leq \sum_{M\geq N} \E\left(|\YM-\E(\YM)|^2\right)/\eta^2.
 \end{align*}
Now we will focus on estimating the variance $Var(\YM)=\E\left(|\YM-\E(\YM)|^2\right).$
Let us observe that 
%\begin{align}
%& Var(\YM)=M^{-2n}\sum_{\substack{ i(1),\dots,i(2n)\in[M] \\ \ker\underline{i} = \pi\\
%k(1),\dots,k(2n)\in[M] \\ \ker\underline{k} = \pi 
% } }\E\Big(\Cof(\pi_f,{\underline{i}})\Cof(\pi_f,{\underline{k}}) \prod_{\substack{\{z,w\} \in \pi\\ f(\{z,w\})=-1}} \Psi_{i(w)}\Psi_{k(w)}\Big)-\E(\YM)^2.
%\label{eq:drugimoment}
%\end{align}
\begin{align}
& Var(\YM)=M^{-2n}\sum_{\substack{ i(1),\dots,i(2n)\in[M] \\ \ker\underline{i} = \pi\\
k(1),\dots,k(2n)\in[M] \\ \ker\underline{k} = \pi 
 } }Cov\Big(\Cof(\pi_f,{\underline{i}})\prod_{\substack{\{z,w\} \in \pi\\ f(\{z,w\})=-1}}\Psi_{i(w)},\Cof(\pi_f,{\underline{k}})  \prod_{\substack{\{z,w\} \in \pi\\ f(\{z,w\})=-1}}\Psi_{k(w)}\Big).
\label{eq:drugimoment}
\end{align}
\textbf{Firstly,} suppose that the set $\{i(1),\dots, i(2n)\}\cap
 \{k (1),\dots, k (2n)\}$ contains at most one element. 
If they are disjoint, then by the independence
assumption (on classical random variable), the corresponding  covariances vanish. 
Assume further that these two sets have exactly one common element. We would like to emphasize
that under this presupposition factors $\Cof(\pi_f,{\underline{i}})$ and $\Cof(\pi_f,{\underline{k}})$ are still independent, so let us consider 
the three different situations. 
\newline
\noindent
Situation 1. 
Assume that our common index corresponds to two blocks of $\pi$, with sign $1$. In this case random variables $\{\Psi_i\}_{i\in \N}$, which appear 
in equation \eqref{eq:drugimoment}, are independent, so the corresponding covariance vanishes.
\newline
\noindent
Situation 2. Assume that our common index corresponds to block of $\pi$ with opposite sign (both of them do not have color $-1$). Then random variables
 $\{\Psi_i\}_{i\in \N}$, which appear in equation \eqref{eq:drugimoment}, are still independent. 
Thus we have analogues to Situation 1, i.e. such indices do not contribute to the variances because of independence.  
\newline
\noindent
Situation 3. In the last case common index corresponds to two blocks of $\pi$, with color $-1$. The modulus of sums over these indices can be bounded by 
$ M^{-2} \NB^2(\pi_f) \sum_{i=1}^\infty Var(\Psi_i)$.  The main points of explanation of this 	
boundary are following. 
Let us assume that indices $\underline{i}$ and $\underline{k}$ with common elements correspond to pair $\{a,a'\}$ and $\{b,b'\}$ (from $\pi$, maybe the same), 
respectively, i.e.  $i(a)=i(a')=k(b)=k(b')$. Then we can bound the corresponding summands as follows
\begin{align*}
& \Big| \sum_{\substack{ i(1),\dots,i(2n)\in[M] \\ \ker\underline{i} = \pi\\
k(1),\dots,k(2n)\in[M] \\ \ker\underline{k} = \pi 
 } }Cov(\cdot)\Big|= \Big|\sum_{i(a)=1}^M\sum_{\substack{i(1),\dots, \check{i}(a),\dots,\check{i}(a'),\dots, i(2n)\in[M] \\ \ker\underline{i} = \pi\\
k(1),\dots,\check{k}(b),\dots,\check{k}(b'),\dots,k(2n)\in[M] \\ \ker\underline{k} = \pi 
 } }Cov(\cdot)\Big|=
 \intertext{by  independece  }
& \Big|\sum_{i(a)=1}^M\sum_{\substack{ \check{i}(a),\dots,\check{i}(a')\in[M] \\ \ker\underline{i} = \pi\\
\check{k}(b),\dots,\check{k}(b')\in[M] \\ \ker\underline{k} = \pi 
 } }\E\Big(\Cof(\pi_f,{\underline{i}})\Cof(\pi_f,{\underline{k}})\prod_{\substack{\{z,w\} \in \pi\setminus\{\{a,a'\}\}\\ f(\{z,w\})=-1}}\Psi_{i(w)}  \prod_{\substack{\{z,w\} \in \pi\setminus\{\{b,b'\}\}\\ f(\{z,w\})=-1}}\Psi_{k(w)}\Big)Var(\Psi_{i(a)})\Big|\\
 &=\Big|\sum_{\substack{ \check{i}(a),\dots,\check{i}(a')\in[M] \\ \ker\underline{i} = \pi\\
\check{k}(b),\dots,\check{k}(b')\in[M] \\ \ker\underline{k} = \pi 
 } }\alpha^{2\NB(\pi_f)-2}q^{2\rc(\pi)+4 \InNB(\pi_f)}\sum_{i(a)=1}^M Var(\Psi_{i(a)})\Big|\leq M^{2n-2} \sum_{i=1}^\infty Var(\Psi_i),
\end{align*}
where for the reader's convenience we skip some indices and expressions in notation (these omissions follow from the context of the proof and we hope that 
this is not an impediment). Factor $\NB^2(\pi_f)$, which appears in the upper bound follows from counting how many times we can assign common indices to a 
block with color $-1$.
Finally, we show that if we have exactly two common elements, then there exists
a $C_1$ (independent from $M$) such that our variance is less than $ M^{-2}C_1$. 

\noindent
\textbf{Secondly,}  it remains to consider the $M^{2n-2}$ rest terms of the sums \eqref{eq:drugimoment}, which by the Cauchy-Bunyakovsky-Schwarz inequality 
is less than $M^{-2n}M^{2n-2}\times C _2.$ 
Summarizing these two estimations above we get 
\begin{align*}
 \begin{split}\E(|\YM-\E(\YM)|^2)\leq M^{-2}\times (C_1+C _2). 
 \end{split}
\end{align*} 
Since the series  $\sum_{M=0}^\infty M^{-2}$ converges, we have $\lim_{N\to \infty}\sum_{M\geq N} \E\left(|\YM-\E(\YM)|^2\right)=0$ and therefore 
$$\Pro\Big(\bigcap_{N\geq 1}\bigcup_{M\geq N} \{|\YM-\E(\YM)|>\eta\}\Big)=\lim_{N\to \infty}\Pro\Big(\bigcup_{M\geq N} \{|\YM-\E(\YM)|>\eta\}\Big)=0,$$
which finishes the proof.\end{proof}
\noindent
The main result of the paper is the following corollary, which refers to the goal specified in Subsection \ref{subgoal}. 
\begin{Cor} \label{corassyptoptic}
Combining Theorem \ref{TwCLT} with Proposition \ref{prop:stochastycznainteroplacja}, and comparing the resulting moments with those given in Subsection \ref{sectionmoment},
immediately yields the desired asymptotic models for the  field operators on the $(\alpha,q)$--Gaussian operator.
Indeed, letting $Z_N = \Su_N + \Su_N^\ast+\Su_N^\prime$, $\E(\Psi_i)=\alpha\langle e,\overline{e}\rangle$ and running over all $\e(i)\in\{\ast,\prime,\1\}$ in 
Theorem \ref{TwCLT} yields
\begin{align*}
\lim_{N\to \infty}\state_N(Z_N^{2n})&=\sum_{\pi_f\in\PB_{2}(2n)}\alpha^{\NB(\pi_f)}q^{\rc(\pi)+2 \InNB(\pi_f)} \prod_{\substack{\{i,j\} \in \pi\\ f(\{i,j\})=-1}} \langle e,\overline{e}\rangle=\langle\Omega, \G^{2n}(e)\Omega\rangle_{\alpha,q}.
%\\&=\langle\Omega, \G(e_{2n})\cdots \G(e_1)\Omega\rangle
\end{align*}
where $e$ is the element of the orthonormal basis of $H_\R.$ \end{Cor} 
\begin{remark}
(1). (The analog of Theorem 2 from \cite{S92}). In order to asymptotically realize the joint moments of $\G(e_{1})\cdots \G(e_{k})$ rather than the moments
of $\G(e)$ alone, it suffices to consider for $\e\in\{\ast,\1,\prime\}$ a sequence 
$$\Su_{N,k}^\e=\frac{1}{\sqrt{N}}\sum_{i=N(k-1)+1}^{Nk}\T^\e_{Nk,i},\quad k\in[N].$$ 
It is a partial sum built from non-intersecting subsets of $\T^\e$. Then we have for all $k\in \N$,  $i(1),\dots,i(k)\in\N$ and $Z_{N,k} = \Su_{N,k} + \Su_{N,k}^\ast+\Su_{N,k}^\prime$  
the following
\begin{align}
&\lim_{N\to \infty}\state_{N^2}(Z_{N,i(1)}\dots Z_{N,i(k)})=\langle\Omega, \G(e_{i(1)})\cdots \G(e_{i(k)})\Omega\rangle_{\alpha,q},
\end{align}
 where  $e_i$ is the element of the orthogonal basis of $H_\R$, which satisfies  $\langle e_i,\overline{e}_j\rangle=0$ for $i\neq j$ and  $\E(\Psi_i)=\alpha\langle e_k,\overline{e_k}\rangle$ for $i\in [N(k-1)+1,Nk]$. This result does not follow directly from Theorem \ref{TwCLT} and Proposition \ref{prop:stochastycznainteroplacja}, but the substantiation
 goes along the same lines as the proof of these two results. This is not entirely obvious
but we leave the formal proof to the reader, because it can be obtained
by modifications of results above. We just wanted to emphasize that the main point in the proof is to take care of the domain of the appropriate indices. 
\newline
(2). 
In the article \cite{BEH15} we show that the annihilator operator can be decomposed as $\B(x)= \r_q(x)+ \alpha  \ell_{q}(\bar{x})q^{N-1}$, where $ x \in H$ 
(for definitions of $\r_q$ and $\ell_{q}$ we refer the reader to \cite{BEH15} because the notation is not short).  Theorem \ref{TwCLT} really says something 
more about asymptotic behavior of $\Su_N, \Su_N^\ast$ and $\Su_N^\prime$.  
Namely,  under the assumption of Corollary \ref{corassyptoptic} and for all $k\in \N$, we have 
\begin{align*}
\lim_{N\to \infty}\state(\Su_N^{\e(1)}\dots\Su_N^{\e(k)})&=\sum_{\pi_f\in\PB_{2,\e}(2n)}\alpha^{\NB(\pi_f)}q^{\rc(\pi)+2 \InNB(\pi_f)} \prod_{\substack{\{i,j\} \in \pi\\ f(\{i,j\})=-1}} \langle e,\overline{e}\rangle\\&=\langle\Omega, \D(e)^{\e(1)}\dots\D(e)^{\e(k)}\Omega\rangle_{\alpha,q},
\end{align*}
where $\D^\ast(e)=\B^\ast(e)$, $\D(e)=\r_q(e)$ and $\D^\prime(e)= \alpha  \ell_{q}(\bar{e})q^{N-1}.$ 
\end{remark}
\noindent
\textbf{Open problem.}
 Formulate CLT for the Gaussian operators of type D \cite{BEH17}. This is completely unclear
to us how to modify the results from this article in order to get CLT of type D. 
\begin{center} Acknowledgments
\end{center}
%The author would like to thank Marek Bo\.zejko and Franz Lehner for suggesting topics, several discussions and helpful comments.
 The work was supported by the Austrian Science Fund (FWF) Project No P 25510-N26 and the Narodowe Centrum Nauki grant 2014/15/B/ST1/00064.

\providecommand{\bysame}{\leavevmode\hbox to3em{\hrulefill}\thinspace}

\end{document}